\documentclass[11pt,a4paper]{scrartcl}

\usepackage{mystyle-koma-enhanced-xetex-pdflatex}
\usepackage{comment}
\usepackage{tikz}
\usetikzlibrary{arrows.meta,positioning}
\usepackage{longtable}
\usepackage{booktabs}

\newcommand{\smat}[1]{\mbox{\scriptsize$\setlength{\arraycolsep}{2pt}\begin{pmatrix}#1\end{pmatrix}$}}

\RequirePackage{tikz}
\usetikzlibrary{backgrounds,patterns,positioning,shapes.geometric}

\MSC{05C20, 05C80, 60C05}
\Keywords{Sidorenko's inequality, directed graphs, homomorphism densities, graph limits}

\newcommand{\N}{\mathbb{N}}

\renewcommand{\P}{\mathbb{P}}

\newcommand{\E}{\mathbb{E}}
\newcommand{\0}{\mathbf{o}}
\newcommand{\R}{\mathbb{R}}

\newcommand{\ind}{\mathbb{1}}
\newcommand{\emb}{\mathup{emb}}

\newcommand{\din}{\operatorname{deg}_{\mathsf{in}}}
\newcommand{\dout}{\operatorname{deg}_{\mathsf{out}}}

\title{Maximising homomorphism counts between digraphs}
\author{
Lukas L\"{u}chtrath \orcidlink{0000-0003-4969-806X}\thanks{Weierstrass Institute for Applied Analysis and Stochastics, Anton-Wilhelm-Amo-Str.\ 39, 10117 Berlin, Germany} \\ lukas.luechtrath@wias-berlin.de \\
\and
Christian M\"{o}nch \orcidlink{0000-0002-6531-6482}%
\\ cmoench25@gmail.com
}
\date{March 19, 2026}

\begin{document}

\maketitle

\begin{abstract}
\noindent We prove a Sidorenko-type inequality for directed trees: for every oriented tree $T$ on $k$
vertices and every finite directed graph $G$, the homomorphism count $\hom⁡(T,G)$ is bounded above by the maximum of the two pure star counts $\hom(S_{0,k-1},G)$
and $\hom(S_{k-1,0},G)$. In other words, among all directed trees on $k$ vertices, the pure in- and out-stars maximise the homomorphism count into host digraphs. The proof is purely combinatorial, based on an iterative leaf-reallocation scheme combined with Hölder's inequality. We further investigate the corresponding homomorphism order on directed trees, discuss refinements via tail-truncation and pointwise bounds for rooted host graphs, and record several consequences, e.g.\ for random directed graph models and local weak limits, where the inequality reduces tree statistics to controlled pure in- and out-degree moments.
\end{abstract}

\section{Introduction}\label{sec:intro}

Sidorenko \cite{Sidorenko_1994} introduced a partial order on graphs based on homomorphism counts and established fundamental extremal properties for trees. In modern terminology, the comparison can equivalently be formulated in terms of homomorphism densities, making the framework compatible
with graphon theory and limit objects.

In this paper we study the \emph{directed} analogue of Sidorenko's
tree inequality. Throughout, a \emph{digraph} $G$ means a finite directed simple loopless graph. Occasionally, we distinguish a root vertex $o\in V(G)$ and write $(G,o)$ when we wish to treat $G$ as rooted. For a digraph $G$ and $v\in V(G)$, write $\din(v)$ and $\dout(v)$ for the in- and out-degree of $v$. Let $\hom(G,H)$ denote the number of digraph homomorphisms mapping the digraphs $G$ to the digraph $H$.

\pagebreak

\begin{theorem}[Directed Sidorenko inequality for trees]
\label{thm:main-directed-sidorenko}
Let $T$ be a directed tree on $k$ vertices. Then for every finite directed graph $H$,
\[
\hom(T,H)
\;\le\;
\max\big\{\hom(S_{0,k-1},H),\hom(S_{k-1,0},H)\big\},
\]
where $S_{0,k-1}$ and $S_{k-1,0}$ denote the pure in- and out-star graphs on $k$ vertices.
\end{theorem}

The proof of Theorem~\ref{thm:main-directed-sidorenko}, given in Section~\ref{sec:directedtrees} is purely combinatorial and is based on an iterative leaf-removal scheme combined with Hölder's inequality. We write $T\preceq \widetilde T$ if $\hom(T,H)\le \hom(\widetilde T,H)$ for all finite digraphs $H$. Lov{\'a}sz showed that finite digraphs are determined by their homomorphism profiles up to isomorphy~\cite{Lovasz71}. Thus, $\preceq$ defines a partial order on (isomorphism classes of) fixed size connected digraphs of which Theorem~\ref{thm:main-directed-sidorenko} identifies two maximal elements.

\paragraph{Further contributions of the paper.}
In Section~\ref{sec:homorder} we investigate the partial order $\preceq$ on directed trees more closely. We show by exhaustive enumeration that all eight isomorphism types of $3$-arc directed trees are pairwise incomparable (Proposition~\ref{prop:3arc-incomparability}), with explicit witness hosts recorded in Appendix~\ref{app:witnesses-3arcs-matrices}. More generally, we prove that directed stars of a given size are pairwise incomparable (Lemma~\ref{lem:stars-fixedsize-incomparable}), and conjecture that $\preceq$ restricted to directed trees of any fixed size is discrete. We also introduce a reversal-symmetrised variant $\preceq_{\max}$ in which the pure star becomes the unique maximum (Remark~\ref{rem:discreteness-conjecture}).

Section~\ref{sec:strength} presents two extensions of the main inequality. First, we establish a directed tail-truncated inequality (Theorem~\ref{thm:directed-tail}) in the spirit of Janson--Kurauskas~\cite{JK2023}, bounding the contribution of high-degree root images at the cost of a factor~$2$ due to the loss of endpoint symmetry. Second, we develop a rooted message-passing framework (Propositions~\ref{prop:rooted-message-passing} and~\ref{prop:pointwise-holder-envelope}) that yields \emph{pointwise} upper bounds on rooted homomorphism counts in terms of local in- and out-degree $\ell^p$-sums, interpolating between crude degree bounds and progressively finer control of neighbourhood degree heterogeneity.

In Section~\ref{sec:applications} we record several applications. Combining the tail inequality with Benjamini--Schramm convergence for digraphs, we obtain a directed analogue of Kurauskas' theorem~\cite{Kurau22}: convergence of pure in- and out-degree moments of order~$h{-}1$ is equivalent to convergence of all connected directed subgraph counts of size~$h$ (Theorem~\ref{thm:directed-kurauskas}). We further derive moment bounds for directed tree patterns in local weak limits (Corollary~\ref{prop:branching-moment}), show that the pointwise $\ell^p$-envelope provides finite fractional moments in heavy-tailed network models where integer moments diverge, and note consequences for entropy maximisation in directed exponential random graph models (Corollary~\ref{prop:entropy-domination}). Finally, we translate our results into matrix inequalities bounding $\hom(T,A)$ in terms of row- and column-sum moments for nonnegative matrices (Corollary~\ref{cor:weighted-tree}), extending the path-based bounds of Merikoski--Virtanen~\cite{MerikoskiVirtanen1991} to arbitrary directed trees.

Appendix~\ref{sec:configorder} establishes the analytic counterpart of the combinatorial order: we define configuration product integrals for nonsymmetric kernels (directed graphons), prove $L^1$-continuity and pullback invariance, and show that the analytic order on kernels is equivalent to the combinatorial homomorphism density order on finite digraphs (Proposition~\ref{thm:orders_equiv_density}).

\section{Sidorenko's inequality for directed trees}\label{sec:directedtrees}

In this section we prove Theorem~\ref{thm:main-directed-sidorenko}.  
The argument is similar in spirit to Sidorenko's original proof for undirected trees~\cite{Sidorenko_1994}, but the directed situation requires a form of weighted leaf-reallocation that distinguishes in- and out-degrees. We follow the streamlined approach using only H\"older's inequality on finite sets that was executed for the undirected setting in \cite{LM2026}.

\subsection{Directed stars}

Throughout, $S_{n-k,k}$ denotes the directed star with $n+1$ vertices and $k$ outgoing arcs and $n-k$ incoming arcs.  
That is, $S_{n-k,k}$ consists of a central vertex connected to $n$ leaves, where $k$ leaves are oriented outward and $n-k$ leaves are oriented inward.

\begin{lemma}[Directed stars]
	\label{lem:holder-star}
	Let $H$ be any finite digraph. For every $n \ge k \ge 0$,
	\[
	\operatorname{hom}(S_{n-k,k},H)
	=
	\sum_{v\in V(H)} \din(v)^{n-k} \dout(v)^{k}
	\le
	\left(\sum_{v\in V(H)} \din(v)^n\right)^{\frac{n-k}{n}}
	\left(\sum_{v\in V(H)} \dout(v)^n\right)^{\frac{k}{n}}.
	\]
	In particular,
	\[
	\operatorname{hom}(S_{n-k,k},H)
	\le
	\max\bigl\{ \operatorname{hom}(S_{n,0},H),\; \operatorname{hom}(S_{0,n},H) \bigr\}.
	\]
\end{lemma}

\begin{proof}
Counting the possible choices for the image of first the centre and then the leaves immediately yields
\[
\hom(S_{n-k,k},H)=\sum_{v\in V(H)} \din(v)^{\,n-k}\dout(v)^{\,k}.
\]
The first claimed inequality now follows immediately by applying the H\"older inequality with exponents $p=\nicefrac{n}{n-k}$ and $q=\nicefrac{n}{k}$. For the extremality statement simply use that for any $A,B\ge 0$ and $\theta\in[0,1]$, one has $A^{1-\theta}B^\theta\le \max\{A,B\}$.
\end{proof}

Having identified $S_{n,0}$ and $S_{0,n}$ as the extremal stars among all $S_{n-k,k}$, the remainder of the proof consists of showing that \emph{every directed tree} can be transformed into one of these stars while never decreasing its homomorphism count.

\subsection{Leaf-reallocation}

Let $T$ be a directed tree.  
A vertex is called a \emph{leaf} if it has total degree one. Removing all leaves from $T$ produces a (possibly trivial) directed \emph{skeleton} tree $\textup{sk}(T)$.

\begin{lemma}[Leaf reallocation]\label{lem:leaf-homog}
Let $T$ be a directed tree with $k$ arcs and at least one internal vertex.
Let $T'$ be obtained from $T$ by deleting all leaves of $T$.
Assume $E(T')\neq \emptyset$ and choose two distinct \emph{leaves of $T'$}, say $a,b$.
For $j\in\{a,b\}$ let $i_j$ (resp.\ $o_j$) denote the number of deleted in-leaves (resp.\ deleted out-leaves) of $T$ that were adjacent to $j$ before passing to $T'$, and set $m_j:=i_j+o_j$.

Then there exists a directed tree $T^\ast$ with $k$ arcs such that
\begin{enumerate}[label=(\roman*)]
\item $\hom(T,H)\le \hom(T^\ast,H)$ for every finite digraph~$H$;
\item $T^\ast$ has one more leaf than $T$;
\item all pendant leaves at the surviving branch vertex of $T^\ast$ are of the \emph{same} type (all in-leaves or all out-leaves).
\end{enumerate}
\end{lemma}

\begin{proof}
Fix a host digraph $H$.  Write $m:=m_a+m_b=i_a+o_a+i_b+o_b$ for the total number of deleted leaves at $a$ and~$b$.

Let $T(a,b)$ be the tree obtained from $T$ by deleting all pendant leaves at $a$ and $b$ (keeping $a,b$ themselves).  For $u,w\in V(H)$ set
\[
N(u,w):=\#\{\phi\in\hom(T(a,b),H):\phi(a)=u,\ \phi(b)=w\}.
\]
Each deleted leaf can be mapped independently, giving exactly $\din(u)$ choices per in-leaf and $\dout(u)$ choices per out-leaf at the image vertex.  Hence
\begin{equation}\label{eq:hom-bilinear}
\hom(T,H)=\sum_{u,w\in V(H)} N(u,w)\,
\din(u)^{i_a}\dout(u)^{o_a}\,
\din(w)^{i_b}\dout(w)^{o_b}.
\end{equation}

We bound the right-hand side in three successive applications of H\"older's inequality.

Apply H\"older to \eqref{eq:hom-bilinear} with the joint weight $N(u,w)$ and exponents
\[
p:=\frac{m}{m_a},\qquad q:=\frac{m}{m_b},\qquad \frac1p+\frac1q=1,
\]
treating $f(u,w):=\din(u)^{i_a}\dout(u)^{o_a}$ and $g(u,w):=\din(w)^{i_b}\dout(w)^{o_b}$.
Since $f$ depends only on $u$ and $g$ only on $w$, this yields
\begin{equation}\label{eq:step1}
\hom(T,H)\le
\biggl(\sum_{u} N_a(u)\,\din(u)^{i_a p}\dout(u)^{o_a p}\biggr)^{\nicefrac{1}{p}}
\biggl(\sum_{w} N_b(w)\,\din(w)^{i_b q}\dout(w)^{o_b q}\biggr)^{\nicefrac{1}{q}},
\end{equation}
where $N_a(u):=\sum_w N(u,w)$ and $N_b(w):=\sum_u N(u,w)$ are the marginals.
Note that $i_a p + o_a p = m$ and $i_b q + o_b q = m$, but the individual exponents $i_a p = i_a m/m_a$ and $o_a p = o_a m/m_a$ need not be integers.

Consider the first factor in \eqref{eq:step1}.  Apply H\"older to the sum over $u$ with the weight $N_a(u)$ and exponents
\[
r:=\frac{m_a}{i_a},\qquad s:=\frac{m_a}{o_a},\qquad \frac1r+\frac1s=\frac{i_a+o_a}{m_a}=1,
\]
applied to the two factors $\din(u)^{i_a p}$ and $\dout(u)^{o_a p}$.  Since $i_a p\cdot r = m$ and $o_a p\cdot s = m$, this gives
\[
\sum_u N_a(u)\,\din(u)^{i_a p}\dout(u)^{o_a p}
\le
\biggl(\sum_u N_a(u)\,\din(u)^{m}\biggr)^{i_a/m_a}
\biggl(\sum_u N_a(u)\,\dout(u)^{m}\biggr)^{o_a/m_a}.
\]
Observe that $\sum_u N_a(u)\,\din(u)^{m} = \hom(T_a^{\mathrm{in}},H)$,
where $T_a^{\mathrm{in}}$ is the tree $T(a,b)$ with $m$ pure in-leaves attached to $a$ (and $b$ retains no pendant leaves, hence is a leaf of the full tree).  Likewise $\sum_u N_a(u)\,\dout(u)^{m} = \hom(T_a^{\mathrm{out}},H)$.

Apply the identical argument to the second factor in \eqref{eq:step1} with exponents $r':=m_b/i_b$ and $s':=m_b/o_b$:
\[
\sum_w N_b(w)\,\din(w)^{i_b q}\dout(w)^{o_b q}
\le
\biggl(\sum_w N_b(w)\,\din(w)^{m}\biggr)^{i_b/m_b}
\biggl(\sum_w N_b(w)\,\dout(w)^{m}\biggr)^{o_b/m_b}.
\]

Substituting both bounds into \eqref{eq:step1} and using $\frac{i_a}{m_a}\cdot\frac{1}{p}=\frac{i_a}{m}$ (and similarly for the other three terms), we obtain
\begin{equation}\label{eq:geom-mean}
\hom(T,H)\le
\hom(T_a^{\mathrm{in}},H)^{i_a/m}\;
\hom(T_a^{\mathrm{out}},H)^{o_a/m}\;
\hom(T_b^{\mathrm{in}},H)^{i_b/m}\;
\hom(T_b^{\mathrm{out}},H)^{o_b/m}.
\end{equation}
Since $\frac{i_a+o_a+i_b+o_b}{m}=1$, the right-hand side is a weighted geometric mean, hence
\[
\hom(T,H)\le \max\bigl\{
\hom(T_a^{\mathrm{in}},H),\;
\hom(T_a^{\mathrm{out}},H),\;
\hom(T_b^{\mathrm{in}},H),\;
\hom(T_b^{\mathrm{out}},H)
\bigr\}.
\]
Let $T^\ast$ be the tree attaining this maximum.  By construction, $T^\ast$ has $m$ pendant leaves at either $a$ or $b$, all of the same type (all in-leaves or all out-leaves), and the other vertex $b$ or $a$ has become a genuine leaf.  Hence $T^\ast$ has one more leaf than $T$ and satisfies (i)--(iii).
\end{proof}

Recall that $T\preceq \widetilde T$ if $\hom(T,H)\le \hom(\widetilde T,H)$ for all finite digraphs $H$.

\begin{proof}[Proof of Theorem~\ref{thm:main-directed-sidorenko}]
Let $T$ be a directed tree with $k$ arcs and $H$ a given host graph. If $T$ is a star we are done by Lemma~\ref{lem:holder-star}.
Otherwise, $T'$ is non-empty and contains two distinct leaves; applying Lemma~\ref{lem:leaf-homog} strictly increases the number of leaves while increasing the homomorphism count. Iterating yields a pure star $S$ with $k$ arcs that dominates $T$ in terms of homomorhisms into $H$.
\end{proof}

\section{The homomorphism order on digraphs}\label{sec:homorder}
We now investigate the partial order $\preceq$ more closely, beginning with the smallest non-trivial case. In the undirected case, the star dominates the path among $3$-edge trees; in the directed setting, no such comparability survives. 
\begin{proposition}[Complete incomparability for $3$-arc directed trees]
\label{prop:3arc-incomparability}
Let $\mathcal T_3$ denote the set of all directed trees with three arcs.
For any two non-isomorphic $T,T'\in\mathcal T_3$, there exist host
digraphs $H_>$ and $H_<$ such that
\[
\hom(T,H_>)>\hom(T',H_>),
\qquad
\hom(T,H_<)<\hom(T',H_<).
\]
In particular, the homomorphism order restricted to $\mathcal T_3$ is discrete.
\end{proposition}

The proof is given by exhaustion of all combinations of the eight isomorphism types of loopless directed trees with three arcs in the Appendix~\ref{app:witnesses-3arcs-matrices}.

We next show that fixed size stars are generally incomparable, in particular the maximal pure stars do never cover the mixed stars in the homomorphism order.

\begin{lemma}[Incomparability of directed stars of given size]
\label{lem:stars-fixedsize-incomparable}
Fix an integer $h\ge 1$. Then, for any $a\neq c\in\{0,1,\dots,h\}$, the stars $S_{a,h-a}$ and $S_{c,h-c}$
are incomparable, i.e.\ there exist loopless digraphs $H,H'$ such that
\[
\hom(S_{a,h-a},H)>\hom(S_{c,h-c},H)
\qquad\text{and}\qquad
\hom(S_{a,h-a},H')<\hom(S_{c,h-c},H').
\]
Equivalently, within the family $\{S_{a,h-a}\}_{a=0}^h$ the order has no
nontrivial comparabilities.
\end{lemma}

\begin{proof}
For any loopless digraph $H$ we have the standard identity
\begin{equation}
\label{eq:star-moment-identity}
\hom(S_{a,h-a},H)=\sum_{v\in V(H)} {\din}_H(v)^a\,{\dout}_H(v)^{h-a}.
\end{equation}
Indeed, the image of the centre can be chosen as any $v\in V(H)$; then each of
the $a$ incoming leaves can be mapped independently to an in-neighbour of $v$,
and each of the $h-a$ outgoing leaves independently to an out-neighbour of $v$.

For integers $m,n\ge 0$ define a loopless digraph $H_{m,n}$ as follows: its
vertex set is $\{v\}\cup U\cup W$ with $|U|=m$, $|W|=n$, and the only edges are
$u\to v$ for $u\in U$ and $v\to w$ for $w\in W$. Then
\[
{\din}_{H_{m,n}}(v)=m,\qquad {\dout}_{H_{m,n}}(v)=n,
\]
and every vertex in $U\cup W$ has either $\din=0$ or $\dout=0$. Hence, if
$1\le a\le h-1$ (so that both exponents in \eqref{eq:star-moment-identity} are
positive), only the centre $v$ contributes and we get
\begin{equation}
\label{eq:star-on-Hmn}
\hom(S_{a,h-a},H_{m,n})=m^a n^{h-a}
\qquad (m,n\ge 1).
\end{equation}

\smallskip
\noindent\emph{Case 1: $1\le a,c\le h-1$ and $a\neq c$.}
Fix $m,n\ge 1$. By \eqref{eq:star-on-Hmn},
\[
\frac{\hom(S_{a,h-a},H_{m,n})}{\hom(S_{c,h-c},H_{m,n})}
=\frac{m^a n^{h-a}}{m^c n^{h-c}}
=\Big(\frac{m}{n}\Big)^{a-c}.
\]
If $a>c$, then choosing $m\gg n$ yields
$\hom(S_{a,h-a},H_{m,n})>\hom(S_{c,h-c},H_{m,n})$, while choosing $m\ll n$ yields
the reverse inequality. The case $a<c$ is symmetric. This gives incomparability.

\smallskip
\noindent\emph{Case 2: one of $a,c$ lies in $\{0,h\}$ and $a\neq c$.}
By symmetry it suffices to treat $a=h$ (the purely in-star $S_{h,0}$) and
$c\in\{0,1,\dots,h-1\}$.

First, take the \emph{pure in-host} $H_{m,0}$ with $m\ge 1$. Then every vertex
has $\dout=0$. If $c<h$, then $h-c>0$, so \eqref{eq:star-moment-identity} gives
\[
\hom(S_{c,h-c},H_{m,0})=\sum_{v} \din(v)^c \dout(v)^{h-c}=0,
\]
whereas $\hom(S_{h,0},H_{m,0})=\sum_v \din(v)^h \ge m^h>0$ (the centre has
in-degree $m$). Hence
\[
\hom(S_{h,0},H_{m,0})>\hom(S_{c,h-c},H_{m,0}).
\]

Second, take $H_{m,n}$ with $m,n\ge 1$. If $1\le c\le h-1$, then by
\eqref{eq:star-on-Hmn} and $\hom(S_{h,0},H_{m,n})=m^h$ we have
\[
\frac{\hom(S_{c,h-c},H_{m,n})}{\hom(S_{h,0},H_{m,n})}
=\frac{m^c n^{h-c}}{m^h}=\Big(\frac{n}{m}\Big)^{h-c}.
\]
Choosing $n\gg m$ yields $\hom(S_{c,h-c},H_{m,n})>\hom(S_{h,0},H_{m,n})$.
If $c=0$, then $\hom(S_{0,h},H_{m,n})$ is at least $n^h$ (the centre has
out-degree $n$), so again $n\gg m$ yields $\hom(S_{0,h},H_{m,n})>m^h$.
Thus we also obtain a host where $S_{h,0}$ loses, proving incomparability.

The remaining boundary combinations ($a=0$ or $c=0$) follow by the same argument
with the \emph{pure out-host} $H_{0,n}$.
\end{proof}

\begin{remark}[Discreteness conjecture and reversal-symmetrised order]
\label{rem:discreteness-conjecture}
Our exhaustive computer search established complete incomparability 
for all directed trees with three arcs (Proposition~\ref{prop:3arc-incomparability}), 
but did not produce witnesses for many pairs of $4$-arc directed trees. 
Nonetheless, based on the theoretical results of this section we conjecture that $\preceq$ 
restricted to directed trees of any fixed size is discrete.

A potentially richer order is obtained by symmetrising with respect to 
arc reversal: for directed trees $T,S$ on $k$ vertices, define
\[
T \preceq_{\max} S
\quad:\Longleftrightarrow\quad
\hom(T,H) \le \max\bigl\{\hom(S,H),\, \hom(S^{\mathrm{rev}},H)\bigr\}
\quad\text{for all finite digraphs } H,
\]
where $S^{\mathrm{rev}}$ denotes $S$ with all arcs reversed.
Theorem~\ref{thm:main-directed-sidorenko} shows that every directed tree 
$T$ satisfies $T \preceq_{\max} S_{0,k-1}$ (since $S_{k-1,0} = S_{0,k-1}^{\mathrm{rev}}$), 
so the pure star is the unique maximum of $\preceq_{\max}$ 
(up to reversal). It would be interesting to determine the full 
structure of $\preceq_{\max}$, and in particular whether it admits 
non-trivial comparabilities beyond the star.
\end{remark}

\section{Extensions of Theorem~\ref{thm:main-directed-sidorenko}}\label{sec:strength}
\subsection{Tail truncation}
\label{subsec:tail}

In the undirected setting, Janson and Kurauskas~\cite{JK2023} prove a refined
tail version of Sidorenko's inequality: for a connected rooted graph $(H,o)$
on $h$ vertices, the homomorphisms whose root image has degree $\ge\Delta$
satisfy $\hom_\Delta(H,G)\le \sum_{v} d_v^{h-1}\ind\{d_v\ge\Delta\}$.
We establish a directed analogue using the total degree
$d(v):=\din(v)+\dout(v)$.

For a rooted directed tree $(T,o)$ and a finite digraph $H$, define
\[
\hom_\Delta(T,H)
:=\bigl|\bigl\{\varphi\in\mathrm{Hom}(T,H): d(\varphi(o))\ge\Delta\bigr\}\bigr|.
\]
More generally, for a weight vector
$\boldsymbol\alpha=(\alpha_x)_{x\in V(T)}\in[0,\infty)^{V(T)}$, set
\begin{equation}\label{eq:hom-alpha-dir}
\hom_{\Delta,\boldsymbol\alpha}(T,H)
:=\sum_{\varphi\in\mathrm{Hom}(T,H)}
  \ind\{d(\varphi(o))\ge\Delta\}\,
  \prod_{x\in V(T)} d(\varphi(x))^{\alpha_x}.
\end{equation}

\begin{theorem}[Directed tail inequality] \label{thm:directed-tail}
Let $(T,o)$ be a rooted directed tree on $k\ge 1$ vertices, let $H$ be a
finite digraph, $\Delta\ge 0$, and
$\boldsymbol\alpha\in[0,\infty)^{V(T)}$. Then
\begin{equation}\label{eq:directed-tail}
\hom_{\Delta,\boldsymbol\alpha}(T,H)
\;\le\; 4\sum_{v\in V(H)} d(v)^{k-1+|\boldsymbol\alpha|}\,\ind\{d(v)\ge\Delta\},
\end{equation}
where $|\boldsymbol\alpha|:=\sum_{x\in V(T)}\alpha_x$. In particular,
$\hom_\Delta(T,H)\le 2\sum_v d(v)^{k-1}\ind\{d(v)\ge\Delta\}$.
\end{theorem}
\begin{proof}
The proof adapts the double induction of \cite[Theorem~3]{JK2023}
over the pair (number of vertices $k$, number of non-root vertices with
positive weight), with a single modification in Case~3 below.
We may assume $T$ is a tree and prove
\eqref{eq:directed-tail}. The base case $k=1$ is immediate.

\smallskip\noindent\emph{Case~1 (leaf with zero weight).}
Suppose $T$ has a leaf $w\neq o$ with $\alpha_w=0$, adjacent to~$p$.
Let $T':=T\setminus\{w\}$ and $\alpha'_p:=\alpha_p+1$,
$\alpha'_u:=\alpha_u$ for $u\neq p$.
The number of extensions from $p$ to $w$ is $\din(\varphi(p))$ or
$\dout(\varphi(p))$ (depending on arc orientation), each bounded by
$d(\varphi(p))$.
Hence $\hom_{\Delta,\boldsymbol\alpha}(T,H)\le
\hom_{\Delta,\boldsymbol\alpha'}(T',H)$ and the bound follows by
induction (fewer vertices).

\smallskip\noindent\emph{Case~2 (two weighted non-roots).}
If two distinct non-roots $v,w$ satisfy $\alpha_v,\alpha_w>0$:
H\"older's inequality with exponents
$(\alpha_v+\alpha_w)/\alpha_v$ and $(\alpha_v+\alpha_w)/\alpha_w$
applies exactly as in \cite[Case~2]{JK2023}, with vertex degrees
replaced by $d(\cdot)$.
This reduces to two subproblems with fewer positive non-root weights;
the weighted geometric--arithmetic mean preserves the constant.

\smallskip\noindent\emph{Case~3 (orientation of a path with at most one weighted non-root).}
If neither Case~1 nor Case~2 applies, then $T$ is an \emph{orientation of a path}
$o=x_0\text{---}x_1\text{---}\cdots\text{---}x_\ell=v$ in the underlying undirected
sense, and $\alpha_{x_i}=0$ for $0<i<\ell$, i.e.\ only the endpoints $o$ and $v$
may carry positive weight.

In the undirected setting, the argument in~\cite{JK2023} uses a Chebyshev
rearrangement based on the symmetry $N(x,y)=N(y,x)$, which need not hold for
directed paths. We proceed by a one-sided rearrangement inequality.

For $x,y\in V(H)$ define
\[
N(x,y):=\#\{\varphi\in\mathrm{Hom}(T,H):\varphi(o)=x,\ \varphi(v)=y\}.
\]
Set
\[
f(x):=d(x)^{\alpha_o}\ind\{d(x)\ge\Delta\},\qquad g(x):=d(x)^{\alpha_v}.
\]
Since $f$ and $g$ are non-decreasing functions of the scalar $d(\cdot)$, we have
\[
(f(x)-f(y))(g(x)-g(y))\ge 0 \qquad \text{for all }x,y\in V(H).
\]
Multiplying by $N(x,y)\ge 0$ and summing over $x,y$ gives
\[
0\le \sum_{x,y\in V(H)} N(x,y)\,(f(x)-f(y))(g(x)-g(y)).
\]
Expanding and rearranging, and writing
\[
R(x):=\sum_{y\in V(H)} N(x,y),\qquad L(y):=\sum_{x\in V(H)} N(x,y),
\]
we obtain
\begin{equation}\label{eq:rearr}
\sum_{x,y} N(x,y)\,f(x)\,g(y)
+\sum_{x,y} N(x,y)\,f(y)\,g(x)
\;\le\;
\sum_{x} R(x)\,f(x)\,g(x)
+\sum_{y} L(y)\,f(y)\,g(y).
\end{equation}
Observe that
\[
\sum_{x,y} N(x,y)\,f(x)\,g(y)=\hom_{\Delta,\boldsymbol\alpha}(T,H)
\]
by definition, since for homomorphisms with $\varphi(o)=x$ and $\varphi(v)=y$
the weight contribution is exactly $f(x)g(y)$. Moreover, the second term on the
left-hand side of~\eqref{eq:rearr} is nonnegative, as all summands are.
We set
\[
A:=\sum_{x} R(x)\,f(x)\,g(x),
\qquad
B:=\sum_{y} L(y)\,f(y)\,g(y),
\]
and proceed by bounding $A$ and $B$ by the induction hypothesis.

\smallskip
\noindent\emph{Bounding $A$.}
Define a new weight vector $\boldsymbol\alpha'$ by
\[
\alpha'_o:=\alpha_o+\alpha_v,\qquad \alpha'_w:=0\ \text{for all }w\neq o.
\]
Then
\[
A=\sum_x R(x)\,d(x)^{\alpha_o+\alpha_v}\ind\{d(x)\ge\Delta\}
=\hom_{\Delta,\boldsymbol\alpha'}(T,H).
\]
Since $|\boldsymbol\alpha'|=|\boldsymbol\alpha|$, $T$ has $k$ vertices, and
$\boldsymbol\alpha'$ has no positive non-root weights, the induction
hypothesis (at the level of the ``in particular'' bound, which holds with
constant~$2$ since every directed homomorphism is an undirected one and the
undirected tail inequality of~\cite{JK2023} gives constant~$1$) yields
\[
A\le 2\sum_{z\in V(H)} d(z)^{k-1+|\boldsymbol\alpha|}\ind\{d(z)\ge\Delta\}.
\]

\smallskip
\noindent\emph{Bounding $B$.}
Reroot $T$ at $v$ and define $\boldsymbol\alpha''$ on $(T,v)$ by
\[
\alpha''_v:=\alpha_o+\alpha_v,\qquad \alpha''_w:=0\ \text{for all }w\neq v.
\]
Since $L(y)$ counts homomorphisms with $\varphi(v)=y$, we have
\[
B=\sum_y L(y)\,d(y)^{\alpha_o+\alpha_v}\ind\{d(y)\ge\Delta\}
=\hom_{\Delta,\boldsymbol\alpha''}((T,v),H).
\]
Applying the induction hypothesis to the rooted tree $(T,v)$ gives
\[
B\le 2\sum_{z\in V(H)} d(z)^{k-1+|\boldsymbol\alpha|}\ind\{d(z)\ge\Delta\}.
\]

Finally, combining these bounds with~\eqref{eq:rearr} and dropping the
non-negative cross term yields
\[
\hom_{\Delta,\boldsymbol\alpha}(T,H)
\le A+B
\le 4\sum_{z\in V(H)} d(z)^{k-1+|\boldsymbol\alpha|}\ind\{d(z)\ge\Delta\},
\]
which is~\eqref{eq:directed-tail}.
\end{proof}

\begin{remark}\label{rem:tail-constant}
In the undirected case, the constant is~$1$
(cf.~\cite[Theorem~3]{JK2023}), thanks to the endpoint symmetry of
undirected paths. The factor~$4$ in~\eqref{eq:directed-tail} arises
because directed paths lack this symmetry. It would be interesting to
determine whether this factor can be removed.
\end{remark}
\subsection{Rooted message passing for pointwise bounds}
\label{subsec:rooted-message-passing}

For a directed graph $G=(V,E)$, write
\[
(A_{\mathrm{out}} f)(v):=\sum_{v\to u} f(u),
\qquad
(A_{\mathrm{in}} f)(v):=\sum_{u\to v} f(u),
\qquad v\in V,
\]
for the (directed) out-/in-neighbourhood summation operators.

Let $(T,o)$ be a rooted directed tree. For each vertex $x\in V(T)$, let
\[
\mathrm{ch}^{\mathrm{out}}(x):=\{c: x\to c \text{ is an arc of }T\},\qquad
\mathrm{ch}^{\mathrm{in}}(x):=\{c: c\to x \text{ is an arc of }T\},
\]
denote the children of $x$ oriented away from / towards $x$.
(Thus every neighbour $c$ of $x$ that lies below $x$ in the rooted tree is in exactly one of these sets.)

\begin{proposition}[Rooted message passing identity]
\label{prop:rooted-message-passing}
Let $(T,o)$ be a rooted directed tree. There exist functions
\(
F_x:V(G)\to \mathbb{R}_{\ge0}
\) for $x\in V(T)$ such that
\[
\hom((T,o),(G,v)) \;=\; F_o(v),\qquad v\in V(G),
\]
and they are computed recursively as follows:
\begin{itemize}
\item If $x$ is a leaf of $T$, then $F_x\equiv 1$.
\item Otherwise, for all $v\in V(G)$,
\begin{equation}
\label{eq:message-passing-recursion}
F_x(v)
=
\prod_{c\in \mathrm{ch}^{\mathrm{out}}(x)} (A_{\mathrm{out}}F_c)(v)
\;\cdot\;
\prod_{c\in \mathrm{ch}^{\mathrm{in}}(x)} (A_{\mathrm{in}}F_c)(v).
\end{equation}
\end{itemize}
\end{proposition}

\begin{proof}
Fix $v\in V(G)$. Every homomorphism $\varphi$ with $\varphi(o)=v$ is uniquely determined by the images of the rooted subtrees below $o$.
At a general vertex $x$, once $\varphi(x)$ is fixed, each outgoing child $c\in\mathrm{ch}^{\mathrm{out}}(x)$ may be mapped to an out-neighbour of $\varphi(x)$, and each incoming child $c\in\mathrm{ch}^{\mathrm{in}}(x)$ to an in-neighbour, independently across different children (the subtrees are disjoint).
This yields \eqref{eq:message-passing-recursion} by conditioning on $\varphi(x)$ and summing over admissible images of the children; leaves contribute the constant $1$.
\end{proof}

For $p\in[1,\infty)$ and $f:V(G)\to\mathbb{R}_{\ge0}$ define the local $p$-sums
\[
\|f\|_{\mathrm{out},p}(v):=\Big(\sum_{v\to u} f(u)^p\Big)^{1/p},
\qquad
\|f\|_{\mathrm{in},p}(v):=\Big(\sum_{u\to v} f(u)^p\Big)^{1/p}.
\]
Then for every $p\in[1,\infty)$ and $v\in V(G)$, Hölder gives the pointwise bounds
\begin{equation}
\label{eq:local-holder}
(A_{\mathrm{out}}f)(v)\le \dout(v)^{1-1/p}\,\|f\|_{\mathrm{out},p}(v),
\qquad
(A_{\mathrm{in}}f)(v)\le \din(v)^{1-1/p}\,\|f\|_{\mathrm{in},p}(v).
\end{equation}

\begin{proposition}[Pointwise Hölder envelope for rooted homomorphism counts]
\label{prop:pointwise-holder-envelope}
Let $(T,o)$ be a rooted directed tree and let $G$ be a directed graph.
Choose exponents $p_{x\to c}\in[1,\infty)$ for every arc $x\to c$ of $T$ (equivalently for each parent--child relation in the rooted tree).
Define $F_x$ as in Proposition~\ref{prop:rooted-message-passing}.
Then for every $v\in V(G)$ and every non-leaf $x\in V(T)$,
\begin{equation}
\begin{aligned}
\label{eq:envelope-step}
F_x(v)
&\le
\Bigg(
\prod_{c\in \mathrm{ch}^{\mathrm{out}}(x)}
\dout(v)^{1-\nicefrac{1}{p_{x\to c}}}
\,
\|F_c\|_{\mathrm{out},p_{x\to c}}(v)
\Bigg)
\cdot
\Bigg(
\prod_{c\in \mathrm{ch}^{\mathrm{in}}(x)}
\din(v)^{1-\nicefrac{1}{p_{c\to x}}}
\,
\|F_c\|_{\mathrm{in},p_{c\to x}}(v)
\Bigg).
\end{aligned}
\end{equation}
In particular, applying \eqref{eq:envelope-step} iteratively from the leaves to the root yields a completely explicit pointwise upper bound on
\(\hom((T,o),(G,v))=F_o(v)\)
in terms of $\din,\dout$ and nested local $p$-sums over directed neighbourhoods of radius at most $\mathrm{ht}(T)$.
\end{proposition}

\begin{proof}
Fix $x$ and $v\in V(G)$. Starting from the recursion \eqref{eq:message-passing-recursion}, apply \eqref{eq:local-holder} to each factor
$(A_{\mathrm{out}}F_c)(v)$ and $(A_{\mathrm{in}}F_c)(v)$ with the corresponding exponent $p_{x\to c}$ (resp.\ $p_{c\to x}$).
Multiplying the resulting bounds gives \eqref{eq:envelope-step}.
The final statement follows by iterating this one-step bound along the rooted tree.
\end{proof}

\begin{corollary}[Uniform-$p$ envelope]
\label{cor:uniform-p-envelope}
In Proposition~\ref{prop:pointwise-holder-envelope}, choose a single exponent $p\in[1,\infty)$ and set
$p_{x\to c}\equiv p$ for every parent--child arc of $T$.
Then for every non-leaf $x$ and every $v\in V(G)$,
\[\begin{split}
F_x(v)
\le
& \dout(v)^{(1-1/p)\,|\mathrm{ch}^{\mathrm{out}}(x)|}
\din(v)^{(1-1/p)\,|\mathrm{ch}^{\mathrm{in}}(x)|}\\
& \times\Bigg(\prod_{c\in \mathrm{ch}^{\mathrm{out}}(x)} \|F_c\|_{\mathrm{out},p}(v)\Bigg)
\Bigg(\prod_{c\in \mathrm{ch}^{\mathrm{in}}(x)} \|F_c\|_{\mathrm{in},p}(v)\Bigg).
\end{split}
\]
In particular, $\hom((T,o),(G,v))=F_o(v)$ is bounded by an expression involving only $\din$,$\dout$ and iterated local $\ell^p$-neighbourhood norms up to depth $\mathrm{ht}(T)$.
\end{corollary}

\begin{remark}[Relation to Theorem~\ref{thm:main-directed-sidorenko}]
\label{rem:envelope-vs-sidorenko}
Proposition~\ref{prop:pointwise-holder-envelope} is a \emph{pointwise} (in $v$) estimate, but it necessarily involves local neighbourhood aggregates (nested $p$-sums) and cannot, in general, be reduced to a function of $\din(v),\dout(v)$ alone. By contrast, the Sidorenko-type inequality in Theorem~\ref{thm:main-directed-sidorenko} is a \emph{global} estimate obtained after averaging over $v$ (and applying Hölder globally), which is precisely why it collapses to mixed in-/out-degree moments. The following example illustrates this. Let $T$ be the directed path $o\to x\to y$, rooted at $o$.
By Proposition~\ref{prop:rooted-message-passing}, for every $v\in V(G)$,
\[
\hom((T,o),(G,v))
=
\sum_{v\to u}\ \sum_{u\to w} 1
=
\sum_{u\in N^{\mathrm{out}}(v)} \dout(u),
\]
where $N^{\mathrm{out}}(v)$ denotes the out-neighbourhood of $v$. Note that $\hom((T,o),(G,v))$ simply is the number of directed walks of length $2$ starting at $v$.

In particular, $\hom((T,o),(G,v))$ cannot be bounded in general by a function of $\din(v)$ and $\dout(v)$ alone: even when $\dout(v)$ is fixed, the quantity above may be arbitrarily large if $v$ has an out-neighbour of very large out-degree.

On the other hand, the uniform-$p$ envelope from Corollary~\ref{cor:uniform-p-envelope} gives the pointwise bound
\[
\hom((T,o),(G,v))
\le
\dout(v)^{1-1/p}\,
\Big(\sum_{v\to u} \dout(u)^p\Big)^{1/p},
\qquad p\in[1,\infty),
\]
which interpolates between the crude estimate
$\hom((T,o),(G,v))\le \sum_{v\to u}\dout(u)$ at $p=1$
and progressively stronger control of large out-neighbour degrees for larger $p$.
\end{remark}

\section{Applications} \label{sec:applications}
\subsection{From degree-moment convergence to LLNs for directed subgraph counts}
\label{subsec:kurauskas-directed}

The main motivation for this paper is a directed analogue of a result of Kurauskas
\cite{Kurau22} showing that, under local weak convergence,
convergence of suitable degree moments is essentially equivalent to convergence of connected
subgraph counts. In the undirected case, the crucial input is Sidorenko's inequality
bounding $\hom(H,G)$ by a star count; see Theorem~2.2 in \cite{Kurau22}.

\begin{proposition}[Moment domination and uniform integrability]
\label{prop:moment-domination}
Let $(G_n)$ be a sequence of finite digraphs with uniformly chosen
root vertex $v_n$. Suppose that
\[
\sup_n \E\big[\din(v_n)^{h-1}\big] < \infty
\qquad\text{and}\qquad
\sup_n \E\big[\dout(v_n)^{h-1}\big] < \infty.
\]
Then for every directed tree $T$ on at most $h$ vertices,
the sequence $n^{-1}\hom(T,G_n)$ is uniformly integrable.
\end{proposition}

\begin{proof}
By Theorem~\ref{thm:main-directed-sidorenko},
\[
n^{-1}\hom(T,G_n)
\le
\max\bigl\{\E[\din(v_n)^{h-1}],\,\E[\dout(v_n)^{h-1}]\bigr\},
\]
and the right-hand side is bounded by assumption.
\end{proof}

We work with the standard Benjamini--Schramm setup, but for rooted \emph{directed} graphs:
let $\mathcal G^\to_\ast$ be the space of rooted, connected, locally finite digraphs modulo rooted
digraph isomorphism, equipped with the usual local metric $d_{\mathrm{loc}}$ via rooted radius-$r$
balls (defined in the underlying undirected graph, i.e.\ induced subdigraph on vertices within
undirected distance $\le r$ from the root, with inherited arc orientations).
For a finite digraph $G$, let $v^\ast$ be a uniform vertex of $G$ and view $(G,v^\ast)$ as a random
element of $\mathcal G^\to_\ast$. We write
\[
(G_n,v_n^\ast)\ \xRightarrow[n\to\infty]{d_{\mathrm{loc}}}\ G^\ast
\]
for local weak convergence in distribution.

\begin{remark}
 Note that in the directed case there are multiple natural ways to define neighbourhoods, and therefore the local topology, depending on the particular application, see e.g.\ \cite[Remark 1.9]{vdHP2025} for a more detailed discussion. We use the most basic notion in which the directions of edges are simply thought of as marks and a digraph thus becomes an instance of a random network in the sense of Aldous and Lyons \cite{AldLyo07}.
\end{remark}

For a rooted digraph $F=(F,r(F))$ and a rooted (finite or infinite) digraph $(G,v)$, let
$\textrm{emb}^\bullet(F,G,v)$ be the number of \emph{injective} rooted homomorphisms
$\varphi:V(F)\to V(G)$ with $\varphi(r(F))=v$ (i.e.\ embeddings preserving all arc orientations).
For an unrooted digraph $H$, let $\mathcal R(H)$ denote the set of rooted versions obtained by
choosing a root in $V(H)$.

\begin{theorem}[Sufficient criterion for subgraph LLN]
\label{thm:directed-kurauskas}
Fix an integer $h\ge2$ and let $\{G_n\}_{n\ge1}$ be a sequence of finite
digraphs with $n_1(n):=|V(G_n)|\to\infty$. Let $v_n^\ast$ be uniform in
$V(G_n)$ and suppose
\[
(G_n,v_n^\ast)\ \xRightarrow[n\to\infty]{d_{\mathrm{loc}}}\ (G^\ast,r^\ast), 
\qquad r^\ast:=\textrm{root}(G^\ast).
\]
Write $D_n:=d_{G_n}(v_n^\ast)={\din}_{G_n}(v_n^\ast)+{\dout}_{G_n}(v_n^\ast)$
for the total degree of the uniform vertex, and
$D_\ast:=d_{G^\ast}(r^\ast)$ for the total degree of the root of the limit.
Assume
\begin{equation}\label{eq:dir-moment-finite-new}
\E\big[D_\ast^{h-1}\big]\ <\ \infty.
\end{equation}
Then the following are equivalent:
\begin{enumerate}[label=(\roman*)]
\item $\E\big[D_n^{h-1}\big]\longrightarrow\E\big[D_\ast^{h-1}\big]$.
\item The family $\{D_n^{h-1}\}_{n\ge1}$ is uniformly integrable.
\item For every connected digraph $H$ on $h$ vertices and every rooted
version $H^\bullet\in\mathcal R(H)$,
\begin{equation}\label{eq:dir-subgraph-conv-new}
n_1(n)^{-1}\,\textrm{emb}(H,G_n)\ \longrightarrow\
\E\big[\textrm{emb}^\bullet(H^\bullet,G^\ast,r^\ast)\big].
\end{equation}
\end{enumerate}
\end{theorem}

\begin{proof}[Proof sketch]
The equivalence (i)$\Leftrightarrow$(ii) follows from
$D_n\Rightarrow D_\ast$ under local weak convergence and the
standard characterisation of uniform integrability via convergence
of expectations.

For (ii)$\Rightarrow$(iii), the argument follows
Kurauskas~\cite[Theorem~2.1]{Kurau22} (see also
Janson--Kurauskas~\cite[Theorem~7]{JK2023} for a streamlined version
that does not require local weak convergence for the upper bound).
Fix a connected digraph $H$ on $h$ vertices, choose a spanning tree
$T$ of the underlying undirected graph of $H$ with inherited arc
orientations, and root $T$ at the root of $H^\bullet$.

The key domination step uses
Theorem~\ref{thm:directed-tail}: for any positive integer $\Delta$,
\[
\E[X_n]\ \le\ \E[\bar X_n]+h\cdot 2\,\E\big[D_n^{h-1}\ind{D_n\ge\Delta}\big],
\]
where $X_n:=\emb^\bullet(H^\bullet,G_n,v_n^\ast)$ and
$\bar X_n$ is the truncated count restricting all internal vertices
to have total degree~$<\Delta$.
Since $0\le \bar X_n\le (\Delta-1)^{h-1}$ is bounded, and
$X_n\xrightarrow{d}X_\ast:=\emb^\bullet(H^\bullet,G^\ast,r^\ast)$ by
local weak convergence, we obtain
\[
\limsup_{n\to\infty}\E[X_n]
\le \E[X_\ast]+h\cdot 2\,\limsup_{n\to\infty}
\E\big[D_n^{h-1}\ind{D_n\ge\Delta}\big].
\]
Letting $\Delta\to\infty$, the tail term vanishes by~(ii),
giving $\limsup\E[X_n]\le \E[X_\ast]$.
The reverse inequality $\E[X_\ast]\le \liminf \E[X_n]$ follows
from Fatou's lemma. Hence $\E[X_n]\to\E[X_\ast]$, which yields
\eqref{eq:dir-subgraph-conv-new}.

For (iii)$\Rightarrow$(i), apply~\eqref{eq:dir-subgraph-conv-new}
to the pure in-star $S_{h-1,0}$ and pure out-star $S_{0,h-1}$,
noting that $$n_1(n)^{-1}\hom(S_{h-1,0},G_n)=\E[\din(v_n^\ast)^{h-1}]$$
and $$\din^{h-1}+\dout^{h-1}\le D^{h-1}\le 2^{h-1}(\din^{h-1}+\dout^{h-1}).$$
\end{proof}

\subsection{Further probabilistic applications}
\label{subsec:probabilistic-applications}

The directed tree inequality of Theorem~\ref{thm:directed-tail} admits a natural probabilistic interpretation.
In locally tree-like directed networks, homomorphism counts of finite
directed trees correspond to expected pattern counts in branching-type
exploration processes.
We briefly record two consequences.

\subsubsection*{Directed local weak limits and branching processes}

Let $(G_n,v_n^\ast)\Rightarrow G^\ast$ in the directed local weak sense
as in Theorem~\ref{thm:directed-kurauskas}.
Write $r^\ast$ for the root of $G^\ast$ and
$D_\ast:=d_{G^\ast}(r^\ast)$ for the total degree at the root.

\begin{corollary}[Moment bound for directed tree counts]
\label{prop:branching-moment}
If\/ $\E[D_\ast^{h-1}]<\infty$, then for every directed tree $T$
on $h$ vertices,
\[
\E\big[\emb^\bullet(T,G^\ast,r^\ast)\big]
\;\le\;
2\,\E\big[D_\ast^{h-1}\big].
\]
In particular, finiteness of the $(h{-}1)$-th total-degree moment
is sufficient for finiteness of the expected number of rooted
directed tree embeddings of size $h$ in the local weak limit.
\end{corollary}
\begin{proof}
By definition of local weak convergence,
$\E[\emb^\bullet(T,G^\ast,r^\ast)]
=\lim_{n\to\infty}\E[\emb^\bullet(T,G_n,v_n^\ast)]$.
Since $\emb^\bullet(T,G_n,v_n^\ast)\le\hom((T,o),(G_n,v_n^\ast))$,
Theorem~\ref{thm:directed-tail} (with $\Delta=0$) yields
$\E[\emb^\bullet(T,G_n,v_n^\ast)]\le 2\,\E[D_n^{h-1}]$.
Passing to the limit gives the claim.
\end{proof}

This statement can be viewed as a mixed-moment
upper bound for pattern counts in multi-type
Galton–Watson trees arising as local weak limits
of directed configuration models or preferential
attachment digraphs, see e.g.\ \cite{CooperFrieze2004,Perarnau2021,Perarnau2023}.

\subsubsection*{Exploration trees of directed random walks}

Let $G$ be a finite directed graph.
Consider the exploration tree generated by a directed
random walk started from a uniformly chosen vertex,
where at each step one selects uniformly among out-neighbours.

For a fixed directed tree $T$ on $k$ vertices,
the number of labelled occurrences of $T$
inside the exploration process is bounded by $\hom(T,G)$.
Hence Theorem~\ref{thm:main-directed-sidorenko} implies
\[
\#\{T\text{-patterns in }G\}
\;\le\;
\max\Big\{\sum_{v\in V(G)}\din(v)^{k-1},\;
\sum_{v\in V(G)}\dout(v)^{k-1}\Big\}.
\]
Consequently,
\[
\mathbb P(\text{exploration tree contains }T)
\;\le\;
\frac{1}{|V(G)|}
\max\Big\{\sum_{v}\din(v)^{k-1},\;\sum_{v}\dout(v)^{k-1}\Big\}.
\]
Thus pure in- and out-degree moments control the
frequency of oriented tree patterns in directed
exploration processes, without any spectral
assumptions on the adjacency matrix.

\subsubsection*{Probabilistic bounds based on pointwise $\ell^p$-envelope}
We illustrate a setting in which the uniform-$p$ envelope from
Corollary~\ref{cor:uniform-p-envelope} is quantitatively useful, whereas the
``$p=1$'' representation is not.

Let $(T,o)$ be the rooted directed path $o\to x\to y$ as in 
Remark~\ref{rem:envelope-vs-sidorenko}, which gives
\[
\hom((T,o),(G,v))=\sum_{v\to u}\dout(u).
\]
Consider a random directed graph model in which, conditional on $\dout(v)=d$,
the out-neighbour degrees $\{\dout(u):v\to u\}$ are i.i.d.\ with a heavy-tailed
distribution $D$ satisfying
\[
\mathbb E[D]=\infty
\qquad\text{but}\qquad
\mathbb E[D^q]<\infty
\ \text{ for some } q\in(0,1).
\]
(One may think of size-biased neighbour degrees in sparse heavy-tailed network
models.)

Then, conditional on $\dout(v)=d\ge1$,
\[
\mathbb E\big[\hom((T,o),(G,v)) \,\big|\, \dout(v)=d\big]
=
d\,\mathbb E[D]
=
\infty,
\]
so the rooted homomorphism count has infinite conditional mean.

On the other hand, for any $p>1$ with $1/p\in(0,1)$ such that $\mathbb E[D^{1/p}]<\infty$,
the uniform-$p$ envelope yields the pointwise bound
\[
\hom((T,o),(G,v))
\le
\dout(v)^{1-1/p}\Big(\sum_{v\to u}\dout(u)^p\Big)^{1/p}.
\]
Taking conditional expectations and using independence,
\[
\mathbb E\Big[\Big(\sum_{v\to u}\dout(u)^p\Big)^{1/p}\ \Big|\ \dout(v)=d\Big]
\le
d^{1/p}\,\mathbb E[D]=\infty.
\]
However, we can attain a suitable bound for the left hand side using finite conditional \emph{$r$-moments} for suitable $r\in(0,1)$:
indeed, for $r\in(0,1)$, concavity of $t\mapsto t^{r/p}$ gives
\[\begin{aligned}
\mathbb E\Big[\hom((T,o),(G,v))^r \,\Big|\, \dout(v)=d\Big]
& \le
d^{r(1-1/p)}\,
\mathbb E\Big[\Big(\sum_{v\to u}\dout(u)^p\Big)^{r/p}\ \Big|\ \dout(v)=d\Big]\\
&\le
d^{r}\,\mathbb E[D^{r}],
\end{aligned}
\]
which is finite whenever $\mathbb E[D^{r}]<\infty$.

Thus, even in regimes where $\mathbb E[\hom((T,o),(G,v))\,|\,\dout(v)]$ is infinite,
the $\ell^p$-envelope provides a convenient tool to obtain finite fractional
moments of rooted homomorphism counts under mild tail assumptions on neighbour degrees. Concrete instances of such size-biased heavy-tailed neighbour-degree laws arise, for example,
in the reciprocal age-dependent random connection model studied in our recent paper \cite{LM_DARCM}.

\subsection{Entropy maximisation and statistical models}
\label{subsec:entropy-applications}

Directed tree homomorphism counts arise naturally as sufficient
statistics in exponential random graph models (ERGMs).
In the dense scaling, ERGMs admit a variational formulation over
directed graphons, where limiting free energies are obtained via
entropy--energy maximisation; see
Chatterjee--Diaconis~\cite{ChatterjeeDiaconis2013}
for the undirected case and the directed developments in
Aristoff--Zhu~\cite{Zhu2015}
and Yin--Zhu~\cite{Zhu2016}.

Fix $h\ge2$ and consider probability measures $P$ on directed graphs
on $n$ vertices satisfying pure degree-moment constraints of the form
\[
\E_P\Big[\sum_{v} \din(v)^{h-1}\Big]=m_{\mathrm{in}},
\qquad
\E_P\Big[\sum_{v} \dout(v)^{h-1}\Big]=m_{\mathrm{out}}.
\]

\begin{corollary}[Tree statistics are moment-dominated]
\label{prop:entropy-domination}
Under the above constraints, for every directed tree $T$ on $h$ vertices,
\[
\E_P[\hom(T,G)]
\;\le\;
\max\{m_{\mathrm{in}},\,m_{\mathrm{out}}\}.
\]
In particular, any maximum-entropy model whose sufficient statistics
include the pure in- and out-degree moments of order $h-1$
automatically controls all directed tree counts of size~$h$.
\end{corollary}

\begin{proof}
By Theorem~\ref{thm:main-directed-sidorenko},
$\hom(T,G)\le\max\{\hom(S_{h-1,0},G)$ and $$\hom(S_{0,h-1},G)\}
=\max\{\sum_v\din(v)^{h-1},\,\sum_v\dout(v)^{h-1}\}$$
for every digraph~$G$. Taking expectations under~$P$ yields the claim.
\end{proof}

In dense directed ERGMs, entropy maximisers under edge and star-type constraints are known to exhibit structured (typically bipodal) graphons; see Aristoff--Zhu~\cite{Zhu2015} for outward star constraints and Yin--Zhu~\cite{Zhu2016} for reciprocity-based models. Corollary~\ref{prop:entropy-domination} shows that directed tree statistics introduce no additional independent constraints once the pure in- and out-degree moments of the
appropriate order are prescribed. Thus, as in the undirected case, tree-based statistics are redundant from an entropy-maximisation perspective, and genuinely new structural effects must arise from cycle or higher-complexity subgraph constraints.

\subsection{Matrix and walk inequalities}
\label{subsec:matrix-ineq}

Our directed tree inequalities can be stated purely in terms of
row/column sums of a nonnegative matrix. This connects our results to
a classical line of work on inequalities for matrix powers, initiated
by London~\cite{London66} and Hoffman~\cite{Hoffman67},
cf.~\cite{Sidorenko85}, and extended to nonsymmetric matrices by
Merikoski and Virtanen~\cite{MerikoskiVirtanen1991} and, more
recently, T\"aubig~\cite{Taubig13}.

Let $A\in\R_{\ge0}^{n\times n}$ be a nonnegative matrix with row sums
$r_i:=\sum_{j}A_{ij}$ and column sums $c_i:=\sum_{j}A_{ji}$.
Interpret $A$ as a weighted directed graph on $[n]$ and define
\[
\hom(T,A):=\sum_{\varphi:V(T)\to[n]}\
\prod_{(u\to v)\in E(T)} A_{\varphi(u)\varphi(v)}.
\]
For the directed path $P$ on $p+1$ vertices,
$\hom(P,A)=\mathrm{sum}(A^p):=\sum_{i,j}(A^p)_{ij}$.

For paths, Merikoski and Virtanen~\cite{MerikoskiVirtanen1991}
established the sharp bound
\[
\mathrm{sum}(A^p)
\;\le\;
\Big(\sum_{i=1}^n c_i^p\Big)^{1/2}
\Big(\sum_{i=1}^n r_i^p\Big)^{1/2}.
\]
The proof of Theorem~\ref{thm:main-directed-sidorenko} extends to
weighted digraphs (replacing degree counts by weighted row/column sums
throughout) and yields the following generalisation from paths to
arbitrary directed trees, at the cost of a weaker bound.

\begin{corollary}[Weighted directed tree inequality]
\label{cor:weighted-tree}
For every $A\in\R_{\ge0}^{n\times n}$ and every directed tree $T$ on
$k$ vertices,
\[
\hom(T,A)\ \le\
\max\Big\{\sum_{i=1}^n c_i^{k-1},\;\sum_{i=1}^n r_i^{k-1}\Big\}.
\]
\end{corollary}

For paths, this bound is weaker than the Merikoski--Virtanen inequality
(since $\sqrt{XY}\le\max\{X,Y\}$), but it applies to all directed
trees. Whether a bound of the form
$\hom(T,A)\le(\sum c_i^{k-1})^{1/2}(\sum r_i^{k-1})^{1/2}$
holds for general directed trees remains open.

\noindent\textbf{\large Funding acknowledgement.} LL received support from the Leibniz Association within the Leibniz Junior Research Group on \textit{Probabilistic Methods for Dynamic Communication Networks} as part of the Leibniz Competition (grant no.\ J105/2020) and by the Deutsche Forschungsgemeinschaft (DFG, German Research Foundation) under Germany's Excellence Strategy - The Berlin Mathematics Research Center MATH+ (EXC-2046/1, EXC-2046/2, project ID: 390685689) through the project \emph{Information flow \& emergent behavior in complex networks}. CM's research was in parts funded by Deutsche Forschungsgemeinschaft (DFG, German Research Foundation) – SPP 2265 443916008.

\section*{References}
\renewcommand*{\bibfont}{\footnotesize}
\printbibliography[heading = none]

\appendix
\section{Witnesses for universal incomparability in the homomorphism order between 3-arc directed trees}
\label{app:witnesses-3arcs-matrices}

We write $A\parallel B$ if there exist loopless hosts $H_1,H_2$ on $n\le 4$ vertices such that
$\hom(A,H_1)>\hom(B,H_1)$ and $\hom(A,H_2)<\hom(B,H_2)$.
All witnesses below were found by exhaustive enumeration of all loopless hosts on $n\le 4$ vertices using a Python script and verified independently.

Stars (center $v_0$) on vertices $\{v_0,v_1,v_2,v_3\}$:
\[
\begin{aligned}
S_{0,3}&:\ v_0\to v_1,\ v_0\to v_2,\ v_0\to v_3 &\quad
S_{3,0}&:\ v_1\to v_0,\ v_2\to v_0,\ v_3\to v_0\\
S_{1,2}&:\ v_0\to v_1,\ v_0\to v_2,\ v_3\to v_0 &\quad
S_{2,1}&:\ v_1\to v_0,\ v_2\to v_0,\ v_0\to v_3
\end{aligned}
\]
Paths along $v_0\text{-}v_1\text{-}v_2\text{-}v_3$:
\[
\begin{aligned}
P_{+++}&:\ v_0\to v_1\to v_2\to v_3 &\quad
P_{++-}&:\ v_0\to v_1\to v_2\leftarrow v_3\\
P_{+-+}&:\ v_0\to v_1\leftarrow v_2\to v_3 &\quad
P_{-++}&:\ v_0\leftarrow v_1\to v_2\to v_3
\end{aligned}
\]
In each adjacency matrix $A_H$ below, entry $(i,j)$ equals $1$ if $i\to j$ is an arc of $H$
(vertices $0$-indexed).

\begingroup
\small
\setlength{\LTpre}{4pt}
\setlength{\LTpost}{4pt}
\setlength{\tabcolsep}{3pt}
\renewcommand{\arraystretch}{1.2}

\begin{longtable}{%
  >{\raggedright\arraybackslash}p{2.5cm}%
  >{\raggedright\arraybackslash}p{\dimexpr0.5\linewidth-1.95cm\relax}%
  >{\raggedright\arraybackslash}p{\dimexpr0.5\linewidth-1.95cm\relax}%
}
\caption{Strict incomparabilities among $3$-arc directed trees, with explicit witness hosts $H$ (loopless, $n\le 4$), given by adjacency matrices. All witnesses are computationally verified.}
\label{tab:long-witnesses-3arcs}\\
\toprule
Pair $A\parallel B$ &
\mbox{Witness $H_>$:} $\hom(A,H_>)>\hom(B,H_>)$ &
\mbox{Witness $H_<$:} $\hom(A,H_<)<\hom(B,H_<)$ \\
\midrule
\endfirsthead
\toprule
Pair $A\parallel B$ &
\mbox{Witness $H_>$:} $\hom(A,H_>)>\hom(B,H_>)$ &
\mbox{Witness $H_<$:} $\hom(A,H_<)<\hom(B,H_<)$ \\
\midrule
\endhead
\midrule\multicolumn{3}{r}{\emph{continued\ldots}}\\\midrule
\endfoot
\bottomrule
\endlastfoot

$S_{0,3}\parallel S_{3,0}$ &
$n=3$, $\smat{0&0&0\\0&0&0\\1&1&0}$; $8>2$ &
$n=3$, $\smat{0&0&0\\1&0&0\\1&0&0}$; $2<8$ \\[4pt]

$S_{0,3}\parallel S_{1,2}$ &
$n=2$, $\smat{0&0\\1&0}$; $1>0$ &
$n=4$, $\smat{0&1&1&0\\1&0&0&0\\1&0&0&0\\1&0&0&0}$; $11<14$ \\[4pt]

$S_{0,3}\parallel S_{2,1}$ &
$n=2$, $\smat{0&0\\1&0}$; $1>0$ &
$n=4$, $\smat{0&0&0&1\\0&0&0&1\\0&0&0&1\\0&1&1&0}$; $11<20$ \\[4pt]

$S_{0,3}\parallel P_{+++}$ &
$n=2$, $\smat{0&0\\1&0}$; $1>0$ &
$n=4$, $\smat{0&0&0&1\\0&0&1&1\\0&1&0&1\\0&1&1&0}$; $25<28$ \\[4pt]

$S_{0,3}\parallel P_{++-}$ &
$n=2$, $\smat{0&0\\1&0}$; $1>0$ &
$n=3$, $\smat{0&0&1\\0&0&1\\0&1&0}$; $3<4$ \\[4pt]

$S_{0,3}\parallel P_{+-+}$ &
$n=3$, $\smat{0&0&0\\0&0&0\\1&1&0}$; $8>4$ &
$n=3$, $\smat{0&0&0\\1&0&0\\1&0&0}$; $2<4$ \\[4pt]

$S_{0,3}\parallel P_{-++}$ &
$n=2$, $\smat{0&0\\1&0}$; $1>0$ &
$n=4$, $\smat{0&0&0&1\\0&0&1&1\\0&1&0&1\\0&1&1&0}$; $25<26$ \\[4pt]

\midrule

$S_{3,0}\parallel S_{1,2}$ &
$n=2$, $\smat{0&0\\1&0}$; $1>0$ &
$n=4$, $\smat{0&0&0&0\\0&0&0&1\\0&0&0&1\\1&1&1&0}$; $11<20$ \\[4pt]

$S_{3,0}\parallel S_{2,1}$ &
$n=2$, $\smat{0&0\\1&0}$; $1>0$ &
$n=4$, $\smat{0&1&1&1\\1&0&0&0\\1&0&0&0\\0&0&0&0}$; $11<14$ \\[4pt]

$S_{3,0}\parallel P_{+++}$ &
$n=2$, $\smat{0&0\\1&0}$; $1>0$ &
$n=4$, $\smat{0&0&0&0\\0&0&1&1\\0&1&0&1\\1&1&1&0}$; $25<28$ \\[4pt]

$S_{3,0}\parallel P_{++-}$ &
$n=2$, $\smat{0&0\\1&0}$; $1>0$ &
$n=4$, $\smat{0&0&0&0\\0&0&1&1\\0&1&0&1\\1&1&1&0}$; $25<26$ \\[4pt]

$S_{3,0}\parallel P_{+-+}$ &
$n=3$, $\smat{0&0&0\\1&0&0\\1&0&0}$; $8>4$ &
$n=3$, $\smat{0&0&0\\0&0&0\\1&1&0}$; $2<4$ \\[4pt]

$S_{3,0}\parallel P_{-++}$ &
$n=2$, $\smat{0&0\\1&0}$; $1>0$ &
$n=3$, $\smat{0&0&0\\0&0&1\\1&1&0}$; $3<4$ \\[4pt]

\midrule

$S_{2,1}\parallel S_{1,2}$ &
$n=3$, $\smat{0&0&1\\0&0&1\\0&1&0}$; $5>3$ &
$n=3$, $\smat{0&0&0\\0&0&1\\1&1&0}$; $3<5$ \\[4pt]

$S_{1,2}\parallel P_{+++}$ &
$n=3$, $\smat{0&0&0\\0&0&1\\1&0&0}$; $1>0$ &
$n=4$, $\smat{0&1&1&0\\1&0&0&0\\0&1&0&0\\0&1&0&0}$; $8<9$ \\[4pt]

$S_{1,2}\parallel P_{++-}$ &
$n=3$, $\smat{0&1&1\\1&0&0\\0&0&0}$; $5>3$ &
$n=3$, $\smat{0&0&0\\1&0&0\\1&1&0}$; $1<2$ \\[4pt]

$S_{1,2}\parallel P_{+-+}$ &
$n=3$, $\smat{0&0&1\\0&0&1\\1&1&0}$; $10>8$ &
$n=2$, $\smat{0&0\\1&0}$; $0<1$ \\[4pt]

$S_{1,2}\parallel P_{-++}$ &
$n=3$, $\smat{0&1&1\\1&0&0\\0&0&0}$; $5>4$ &
$n=3$, $\smat{0&1&1\\0&0&1\\0&0&0}$; $1<2$ \\[4pt]

\midrule

$S_{2,1}\parallel P_{+++}$ &
$n=3$, $\smat{0&0&0\\0&0&1\\1&0&0}$; $1>0$ &
$n=4$, $\smat{0&1&1&1\\1&0&0&0\\0&1&0&0\\0&0&0&0}$; $8<9$ \\[4pt]

$S_{2,1}\parallel P_{++-}$ &
$n=3$, $\smat{0&1&0\\1&0&0\\1&0&0}$; $5>4$ &
$n=3$, $\smat{0&0&0\\1&0&0\\1&1&0}$; $1<2$ \\[4pt]

$S_{2,1}\parallel P_{+-+}$ &
$n=3$, $\smat{0&0&1\\0&0&1\\1&1&0}$; $10>8$ &
$n=2$, $\smat{0&0\\1&0}$; $0<1$ \\[4pt]

$S_{2,1}\parallel P_{-++}$ &
$n=3$, $\smat{0&1&0\\1&0&0\\1&0&0}$; $5>3$ &
$n=3$, $\smat{0&0&0\\1&0&0\\1&1&0}$; $1<2$ \\[4pt]

\midrule

$P_{+++}\parallel P_{++-}$ &
$n=4$, $\smat{0&0&0&0\\0&0&0&1\\0&1&0&0\\1&1&1&0}$; $9>8$ &
$n=3$, $\smat{0&0&0\\0&0&1\\1&0&0}$; $0<1$ \\[4pt]

$P_{+++}\parallel P_{-++}$ &
$n=4$, $\smat{0&0&0&1\\0&0&0&1\\0&1&0&0\\1&1&0&0}$; $10>9$ &
$n=3$, $\smat{0&0&0\\0&0&1\\1&0&0}$; $0<1$ \\[4pt]

$P_{++-}\parallel P_{+-+}$ &
$n=4$, $\smat{0&0&0&1\\0&0&1&1\\0&1&0&1\\0&1&1&0}$; $32>31$ &
$n=2$, $\smat{0&0\\1&0}$; $0<1$ \\[4pt]

$P_{++-}\parallel P_{-++}$ &
$n=3$, $\smat{0&0&1\\0&0&1\\0&1&0}$; $4>3$ &
$n=3$, $\smat{0&0&0\\0&0&1\\1&1&0}$; $3<4$ \\[4pt]

$P_{+-+}\parallel P_{-++}$ &
$n=2$, $\smat{0&0\\1&0}$; $1>0$ &
$n=4$, $\smat{0&0&0&0\\0&0&1&1\\0&1&0&1\\1&1&1&0}$; $31<32$ \\[4pt]

\end{longtable}
\endgroup

It remains to show the incomparability of $P_{+++}$ and $P_{+-+}$, which requires a $5$-vertex witness. 

For a loopless host digraph $H$ with adjacency matrix $A$, we have
\[
\hom(P_{+++},H)=\sum_{a,b,c,d} A_{ab}A_{bc}A_{cd},
\qquad
\hom(P_{+-+},H)=\sum_{a,b,c,d} A_{ab}A_{cb}A_{cd}.
\]

We exhibit an explicit loopless digraph $H$ on vertex set
$\{0,1,2,3,4\}$ such that
$\hom(P_{+++},H)>\hom(P_{+-+},H)$.
Let $H$ have edge set
\[
E(H)=\{\,0\!\to\!1,\ 0\!\to\!2,\ 1\!\to\!2,\ 1\!\to\!3,\ 1\!\to\!4,\ 2\!\to\!0,\ 2\!\to\!1,\ 3\!\to\!0,\ 4\!\to\!0\,\},
\]
i.e.\ adjacency matrix
\[
A=
\begin{pmatrix}
0&1&1&0&0\\
0&0&1&1&1\\
1&1&0&0&0\\
1&0&0&0&0\\
1&0&0&0&0
\end{pmatrix}.
\]

Write $d^+(v)=\sum_x A_{vx}$ and $d^-(v)=\sum_y A_{yv}$.
A direct calculation gives
\[
\begin{split}
(d^+(0),d^+(1),d^+(2),d^+(3),d^+(4))=(2,3,2,1,1),\\
(d^-(0),d^-(1),d^-(2),d^-(3),d^-(4))=(3,2,2,1,1).
\end{split}
\]

We evaluate the difference
$\Delta(H):=\hom(P_{+-+},H)-\hom(P_{+++},H)$ using the identity
\[
\Delta(H)=\sum_{b,c} d^-(b)\,d^+(c)\,\bigl(A_{cb}-A_{bc}\bigr)
       \;=\;\sum_{x\to y\in E(H)}\Bigl(d^-(y)d^+(x)-d^-(x)d^+(y)\Bigr),
\]
which follows by summing first over the endpoints of the outer arcs in
the $4$-tuple expansions.

For the present $H$, summing the contribution
$d^-(y)d^+(x)-d^-(x)d^+(y)$ over the nine directed edges listed above yields
\[
\Delta(H)=-1.
\]
Equivalently,
\[
\hom(P_{+++},H)=37,
\qquad
\hom(P_{+-+},H)=36,
\]
so $\hom(P_{+++},H)>\hom(P_{+-+},H)$.

On the other hand, taking for instance the loopless host consisting of a
single directed edge, we have $\hom(P_{+++},H)=0$ but $\hom(P_{+-+},H)=1$,
hence $\hom(P_{+++},H)<\hom(P_{+-+},H)$ for that host.

\section{Nonsymmetric kernels, configuration products, and order equivalence}\label{sec:configorder}

Let $(\Omega,\mathcal F,\mu)$ be a probability space. A \emph{kernel} is a
measurable function $h:\Omega^2\to[0,1]$.
For a finite digraph $D$ with vertex set $V(D)=\{v_1,\dots,v_n\}$ and arc set $E(D)$ we define
the (directed) \emph{configuration product integral}
\begin{equation}\label{eq:UD_def}
U_D(h)
:=
\int_{\Omega^n}\ \prod_{(v_i,v_j)\in E(D)} h(x_i,x_j)\, \mathrm d\mu(x_1)\cdots \mathrm d\mu(x_n).
\end{equation}
This is the directed analogue of Sidorenko's functional $U_G$ for graphs, cf.\ \cite{Sidorenko_1994}.

\begin{definition}[Order via configuration products]\label{def:order}
For finite digraphs $D,F$ we write $D\ge F$ if
\begin{equation}\label{eq:order_allspaces}
U_D(h)\ge U_F(h)
\quad\text{for all standard probability spaces }(\Omega,\mathcal F,\mu)\text{ and all kernels }h:\Omega^2\to[0,1].
\end{equation}
\end{definition}

\subsection{Preliminary reductions}

We now use standard methods to show that we may restrict ourselves to finite measure spaces.

\begin{lemma}[Reductions]\label{lem:reductions}
Let $D,F$ be finite digraphs. The following are equivalent.
\begin{enumerate}[label=(\roman*)]
\item \eqref{eq:order_allspaces} holds.
\item The inequality $U_D(h)\ge U_F(h)$ holds for all finite measure spaces $(\Omega,\mathcal F,\mu)$
and all bounded kernels $h:\Omega^2\to[0,\infty)$.
\end{enumerate}
\end{lemma}

\begin{proof} For (i)$\Rightarrow$(ii), let $(\Omega,\mathcal F,\mu)$ have
finite mass and let $h$ be bounded with $0\le h\le M<\infty$. Replace $\mu$ by the probability
measure $\mu':=\mu(\Omega)^{-1}\mu$ and $h$ by $h':=M^{-1}h\in[0,1]$. Then
\[
U_D^{(\mu)}(h)=\mu(\Omega)^{|V(D)|}M^{|E(D)|}U_D^{(\mu')}(h'),
\]
and likewise for $F$. Hence $U_D^{(\mu')}(h')\ge U_F^{(\mu')}(h')$ implies
$U_D^{(\mu)}(h)\ge U_F^{(\mu)}(h)$. Finally, (ii)$\Rightarrow$(i) is trivial by restricting
to probability spaces and kernels in $[0,1]$.
\end{proof}


If $\phi:(\Omega',\mathcal F',\mu')\to(\Omega,\mathcal F,\mu)$ is measure-preserving and
$h$ is a kernel on $\Omega$, define the pullback kernel
\[
h^\phi(x,y):=h(\phi(x),\phi(y))\qquad(x,y\in\Omega').
\]

\begin{lemma}[Pullback invariance]\label{lem:pullback}
For every digraph $D$ and every measure-preserving $\phi$ as above,
\[
U_D(h^\phi)=U_D(h).
\]
\end{lemma}

\begin{proof}
This is a direct change-of-variables statement: under $\mu'^n$, the vector
$(\phi(X_1),\dots,\phi(X_n))$ has law $\mu^n$ whenever $(X_1,\dots,X_n)\sim\mu'^n$.
\end{proof}
Considering bounded kernels straightforwardly yields continuity in $L^1$ of the configuration products. 
\begin{lemma}[{$L^1$-continuity on $[0,1]$-bounded kernels}]\label{lem:L1cont}
Fix a digraph $D$ with $m:=|E(D)|$ arcs. Let $(\Omega,\mathcal F,\mu)$ be a standard probability space.
If $h,h_n:\Omega^2\to[0,1]$ and $\|h_n-h\|_{L^1(\mu\otimes\mu)}\to0$, then
$U_D(h_n)\to U_D(h)$.
\end{lemma}

\begin{proof}
Write $E(D)=\{e_1,\dots,e_m\}$ with $e_r=(a_r,b_r)$. For $\mathbf x=(x_1,\dots,x_n)\in\Omega^n$
set
\[
P(\mathbf x):=\prod_{r=1}^m h(x_{a_r},x_{b_r}),\qquad
P_n(\mathbf x):=\prod_{r=1}^m h_n(x_{a_r},x_{b_r}).
\]
Since all factors lie in $[0,1]$, the telescoping bound gives
\[
|P_n(\mathbf x)-P(\mathbf x)|
\le \sum_{r=1}^m |h_n(x_{a_r},x_{b_r})-h(x_{a_r},x_{b_r})|.
\]
Integrating over $\mu^n$ and using Fubini,
\[
|U_D(h_n)-U_D(h)|
\le \sum_{r=1}^m \int_{\Omega^2}|h_n(x,y)-h(x,y)|\,\mathrm d\mu(x)\mathrm d\mu(y)
= m\,\|h_n-h\|_1\to0.
\]
\end{proof}

\begin{remark}
If one insists on quantifying over \emph{all} probability spaces (not necessarily standard),
one needs a separate measure-theoretic approximation argument. In this paper we work on
standard spaces, which is the natural framework for directed graphons and suffices for the
order theory and all applications.
\end{remark}

\subsection{Equivalence of orders}
For a finite digraph $H$ and a finite digraph $D$ define the homomorphism density
\[
t(D,H):=\frac{\hom(D,H)}{|V(H)|^{|V(D)|}}.
\]
We say that $D$ \emph{dominates} $F$ in the combinatorial sense if $t(D,H)\ge t(F,H)$ for all
finite digraphs $H$. The following reduction to the homomorphism count order of Section~\ref{sec:directedtrees} is standard.

\begin{proposition}[Equivalence of homomorphism and density order]
\label{thm:equiv_orders_standard}
Let $G,L$ be finite simple directed graphs without loops or multiple edges.
The following are equivalent:

\begin{enumerate}
\item For every finite directed graph $H$,
\[
\hom(G,H) \ge \hom(L,H).
\]
\item For every directed graphon $W$,
\[
t(G,W) \ge t(L,W).
\]
\end{enumerate}
\end{proposition}

\begin{proof}
(1) $\Rightarrow$ (2):
Every graphon can be approximated in cut metric by step graphons arising from
finite directed graphs. Since $t(G,\cdot)$ is continuous in the cut metric,
the inequality extends to all graphons.

(2) $\Rightarrow$ (1):
Every finite directed graph $H$ induces a step graphon $W_H$ such that
\[
t(G,W_H)=\frac{\hom(G,H)}{|V(H)|^{|V(G)|}}.
\]
Thus the density inequality implies the homomorphism inequality.
\end{proof}

\begin{proposition}[Equivalence of analytic and combinatorial orders]\label{thm:orders_equiv_density}
Let $D,F$ be finite digraphs. The following are equivalent.
\begin{enumerate}[label=(\roman*)]
\item (\emph{Analytic order}) $U_D(h)\ge U_F(h)$ for all standard probability spaces
$(\Omega,\mathcal F,\mu)$ and all kernels $h:\Omega^2\to[0,1]$.
\item (\emph{Combinatorial order}) $t(D,H)\ge t(F,H)$ for all finite digraphs $H$.
\end{enumerate}
\end{proposition}

\begin{proof}
\emph{(i)$\Rightarrow$(ii).}
Fix a finite digraph $H$ with vertex set $[N]$.
Let $\Omega=[0,1]$ with Lebesgue measure and partition $\Omega$ into intervals
$I_1,\dots,I_N$ of equal length $1/N$.
Define the step kernel
\[
h_H(x,y):=\ind\{\{(i,j)\in E(H)\}\}\qquad\text{whenever } x\in I_i,\ y\in I_j.
\]
Expanding \eqref{eq:UD_def} and using that each $x_r$ falls into each interval with
probability $1/N$, we obtain
\[
U_D(h_H)=\frac{\hom(D,H)}{N^{|V(D)|}}=t(D,H),\qquad
U_F(h_H)=t(F,H).
\]
Hence $U_D(h_H)\ge U_F(h_H)$ implies $t(D,H)\ge t(F,H)$.

\smallskip
\emph{(ii)$\Rightarrow$(i).}
By Lemma~\ref{lem:pullback} and the structure theorem for
standard probability spaces, it suffices to prove (i) for kernels
$h:[0,1]^2\to[0,1]$ on $([0,1],\lambda)$.
Assume for contradiction that there exists such an $h$ and some $\varepsilon>0$ such that
\begin{equation}\label{eq:contrad_eps_density}
U_D(h)\le U_F(h)-3\varepsilon.
\end{equation}

For $n\in\N$ let $G_n=G(n,h)$ be the inhomogeneous random \emph{loopless} digraph on vertex set $[n]$
obtained as follows: sample i.i.d.\ $X_1,\dots,X_n\sim\mathrm{Unif}[0,1]$ and, conditional on the $X_i$,
insert each arc $(i,j)$ with $i\neq j$ independently with probability $h(X_i,X_j)$.

Since $Q$ is loopless, each factor corresponds to an arc $(i,j)$ with $i\neq j$, so the construction above matches \eqref{eq:UD_def} exactly. For every fixed digraph $Q$ we have
\begin{equation}\label{eq:EGn}
\E[t(Q,G_n)]=U_Q(h).
\end{equation}
Moreover, for each fixed $Q$ we claim the concentration
\begin{equation}\label{eq:concGn}
t(Q,G_n)\xrightarrow[n\to\infty]{\P}U_Q(h).
\end{equation}
To prove \eqref{eq:concGn}, expose the labels $X_1,\dots,X_n$ sequentially and consider the Doob
martingale $M_i:=\E[t(Q,G_n)\mid X_1,\dots,X_i]$.
Changing one label $X_i$ can only affect those homomorphism maps $\varphi:V(Q)\to[n]$ for which
$\varphi(v)=i$ for at least one $v\in V(Q)$.
There are at most $|V(Q)|\,n^{|V(Q)|-1}$ such maps, and each contributes at most $1/n^{|V(Q)|}$
to $t(Q,G_n)$.
Hence the martingale differences satisfy $|M_i-M_{i-1}|\le |V(Q)|/n$ almost surely, and
Azuma--Hoeffding yields
\[
\P\big(|t(Q,G_n)-\E t(Q,G_n)|>\varepsilon\big)\le 2\exp\!\Big(-\frac{\varepsilon^2 n}{2|V(Q)|^2}\Big),
\]
which implies \eqref{eq:concGn} together with \eqref{eq:EGn}.

Applying \eqref{eq:concGn} to $Q=D$ and $Q=F$, for all sufficiently large $n$ we have with positive
probability simultaneously
\[
t(D,G_n)\le U_D(h)+\varepsilon,\qquad
t(F,G_n)\ge U_F(h)-\varepsilon.
\]
Combined with \eqref{eq:contrad_eps_density} this implies $t(D,G_n)<t(F,G_n)$ with positive
probability. Hence there exists a deterministic finite digraph $H$ (a realisation of $G_n$)
such that $t(D,H)<t(F,H)$, contradicting (ii).
\end{proof}

\end{document}